\documentclass[12pt,fleqn]{article}

\usepackage{mathtools,amsfonts,amsmath, amsthm,amssymb}

\usepackage[active]{srcltx}

\usepackage{fouriernc}

\usepackage{latexsym}
\usepackage{esint}

\usepackage[margin=22mm]{geometry} 
 \numberwithin{equation}{section}

\usepackage{color}
 \definecolor{db}{rgb}{0.0,0.0,0.8} 
\definecolor{dg}{rgb}{0.0,0.55,0.14}
\definecolor{dr}{rgb}{0.5,0,0.07}

  \usepackage{hyperref}
 \hypersetup{frenchlinks=true,colorlinks=true,linkcolor=db,citecolor=dg,
 bookmarks=true}

 \usepackage{enumerate}
\usepackage{authblk}

\newtheorem{theorem}{Theorem}[section]
\newtheorem{proposition}[theorem]{Proposition}

\newtheorem{lemma}[theorem]{Lemma}
\newtheorem{corollary}[theorem]{Corollary}
\theoremstyle{definition}

\theoremstyle{definition}

\theoremstyle{definition}

\theoremstyle{definition}

\theoremstyle{definition}

\theoremstyle{definition}
\newtheorem{remark}{Remark}
\theoremstyle{definition}

\newtheorem{open-problem}{Open Problem}
\newcounter{step}

\newcommand{\rlemma}[1]{Lemma~\ref{#1}}
\newcommand{\rth}[1]{Theorem~\ref{#1}}
\newcommand{\rcor}[1]{Corollary~\ref{#1}}

\newcommand{\rprop}[1]{Proposition~\ref{#1}}

\def\be{\begin{equation}}
\def\ee{\end{equation}}
\def\bes{\begin{equation*}}
\def\ees{\end{equation*}}
\def\bt{\begin{theorem}}
\def\et{\end{theorem}}
\def\bpr{\begin{proposition}}
\def\epr{\end{proposition}}
\def\bl{\begin{lemma}}
\def\el{\end{lemma}}
\def\bc{\begin{corollary}}
\def\ec{\end{corollary}}
\def\br{\begin{remark}}
\def\er{\end{remark}}
\def\ben{\begin{enumerate}}
\def\bena{\begin{enumerate}[a)]}
\def\een{\end{enumerate}}
\def\bit{\begin{itemize}}
\def\iit{\end{itemize}}

\def\supp{\operatorname{supp}}
\def\dist{\operatorname{dist}}

\def\sgn{\operatorname{sgn}}

 \newcommand{\Prod}{\mathop{\prod}\limits}
\DeclareMathAlphabet{\mathonebb}{U}{bbold}{m}{n}

\def\R{{\mathbb R}}
\def\N{{\mathbb N}}

\usepackage{csquotes}

\def\st{{\mathbb S}^2}

\def\SNO{\omega_{N-1}}

\def\H1O{H^1(\Omega;\st)}
\def\dist{\operatorname{dist}}
\def\Dist{\operatorname{Dist}}

\def\DistH1{\Dist_{H^1}}
\def\distH1{\dist_{H^1}}

\def\sgn{\operatorname{sgn}}

\newcommand{\tu}{\widetilde u}
\newcommand{\tv}{\widetilde v}

\date{\today}
\title{The best constant in the embedding of $W^{N,1}(\R^N)$ into $L^\infty(\R^N)$}
\author{Itai Shafrir \thanks{Department of Mathematics, Technion - I.I.T., 32000 Haifa, Israel. Email address: shafrir$@$math.technion.ac.il}}

\newcommand\blfootnote[1]{%
  \begingroup
  \renewcommand\thefootnote{}\footnote{#1}%
  \addtocounter{footnote}{-1}%
  \endgroup
}

\begin{document}
\maketitle

\begin{abstract}
We compute the best constant in the embedding of $W^{N,1}(\R^N)$ into
$L^\infty(\R^N)$, extending a result of Humbert and Nazaret in
dimensions one and two to
 any $N$. The main tool is the identification of $\log |x|$
as a fundamental solution of a certain elliptic operator of order
$2N$.  
\end{abstract}
\section{Introduction}
It is well known that the space $W^{N,1}(\R^N)$ is continuously
embedded in $L^\infty(\R^N)$. Actually, it is embedded in $C_0(\R^N)$ -- the subspace of $C(\R^N)$ consisting of
functions satisfying $\lim_{|x|\to\infty} u(x)=0$. This follows from  the density of $C^\infty_c(\R^N)$ in
$W^{N,1}(\R^N)$ and the inequality 
\begin{equation}
  \label{eq:60}
  \|u\|_\infty\leq C\int_{\R^N}|\nabla^N u|\,dx\,,~\forall u\in C^\infty_c(\R^N)\,,
\end{equation}
see e.g.~Brezis~\cite[Remark 13, Chapter 9]{br}. 
 Above we used the notation
 \begin{equation}
   \label{eq:61}
   \nabla^N u=\left\{\frac{\partial ^Nu}{\partial
       x_{i_1}\cdots\partial x_{i_N}}\right\}_{i_1,\ldots,i_N\in I_N},
 \end{equation}
 where $I_N:=\{1,\ldots,N\}$ (so that $\nabla^N u$ is a tensor of size
 $N^N$) and
\begin{equation}
  \label{eq:62}
  |\nabla^N u|=\left\{\sum_{i_1,\ldots,i_N\in I_N}\left(\frac{\partial ^Nu}{\partial x_{i_1}\cdots\partial x_{i_N}}\right)^2\right\}^{1/2}\,.
\end{equation}
It is natural to look for the optimal constant $C$ in \eqref{eq:60} (i.e., the
smallest constant for which the inequality holds). Of course it is
equivalent to consider the inequality for either $u\in
C^\infty_c(\R^N)$ or $u\in W^{N,1}(\R^N)$. Following Humbert and Nazaret~\cite{hn}
we denote the optimal  constant in \eqref{eq:60} by $K_N$. The
constant $K_N$ played an important role in their study of best
constants in the inequality
\begin{equation}
  \label{eq:7}
  \|u\|_\infty\leq A \int_M |\nabla^N_g u|_g\,dv_g+B\|u\|_{W^{N-1,1}(M)}\,,\quad\forall u\in C^\infty(M),
\end{equation}
 for a smooth compact Riemannian $N$--manifold $(M,g)$. In 
 \cite{hn} they computed the
value of $K_N$ for $N=1,2$, and left open the question of computing its
value for $N\geq 3$. We answer this question
in the current paper. Our argument is valid for every $N\geq1$.
\begin{remark}
The particular choice of the norm as  in \eqref{eq:62} is important,
since as we shall see below, it is invariant with respect to the
orthogonal group $O(N)$. 
\end{remark}
In order to prescribe the value of
the best constant we recall (see Morii, Sato and Sawano~\cite{mss})
the following property of  the function $\log |x|$:
\begin{equation}
  \label{eq:63}
  |\nabla^N \log |x||=\frac{\sqrt{\ell_N}}{|x|^N},
\end{equation}
for some positive constant $l_N$ (see also \rcor{cor:rad} below). Its
explicit (and complicated!) value
was calculated in \cite{mss} (see Remark~\ref{rem:explicit}
below). Our main result is the following ($\omega_m$ denotes the surface area of
 the unit sphere ${\mathbb S}^m$):
\begin{theorem}
  \label{th:main}
 (i) The value of $K_N$ is  $\left(\sqrt{l_N}\,\SNO\right)^{-1}$. That is, 
 \begin{equation}
   \label{eq:64}
   \|u\|_\infty\leq \left(\sqrt{l_N}\,\SNO\right)^{-1}\int_{\R^N}|\nabla^N
   u|\,dx\,,\quad \forall u\in W^{N,1}(\R^N),
 \end{equation}
  and one cannot replace $\left(\sqrt{l_N}\,\SNO\right)^{-1}$ by a smaller
  constant. \\
(ii) Moreover, for $N\geq 2$ there is no function in
  $W^{N,1}(\R^N)$ (except the zero function) for which equality holds
  in \eqref{eq:64}. 
\end{theorem}
\begin{remark}
  An easy consequence of \rth{th:main} is that the same result holds
  in the space $W^{N,1}_0(\Omega)$ (the closure of $C^\infty_c(\Omega)$
  in $W^{N,1}(\Omega)$), for {\em every}
  $\Omega\subset\R^N$. More precisely,  
  \begin{equation*}
    \|u\|_\infty\leq \left(\sqrt{l_N}\,\SNO\right)^{-1}\int_{\Omega}|\nabla^N
   u|\,dx\,,\quad \forall u\in W^{N,1}_0(\Omega),
  \end{equation*}
 the constant $K_N=\left(\sqrt{l_N}\,\SNO\right)^{-1}$ is optimal, and
 the inequality is strict, unless $u$ is the zero function. This
 follows from the invariance of the $W^{N,1}$-seminorm with respect to scalings.
\end{remark}
 The main ingredient of the proof of \rth{th:main} is another property
 of the function $\log |x|$, namely, it is a fundamental solution for the
 elliptic operator 
$${\mathcal L}=(-1)^N \sum_{i_1,\ldots,i_N\in I_N} \frac{\partial ^N}{\partial x_{i_1}\cdots\partial
  x_{i_N}}\left(|x|^N \frac{\partial ^N}{\partial x_{i_1}\cdots\partial
  x_{i_N}}\right),~\text{ (see \rprop{prop:prop} below)}.
$$
\subsubsection*{Acknowledgments.} The author
was supported by  the Israel Science Foundation (Grant No. 999/13). Part of this work was done while he was visiting the University Claude Bernard Lyon 1. He thanks the Mathematics Department for its hospitality.

\section{Preliminaries}
\label{sec:prelo}
    We will consider also the generalization of
    \eqref{eq:61}--\eqref{eq:62} for any $m\geq1$. So for each $u\in
    C^m(\Omega)$, $m\geq 1$, where $\Omega$ is a domain in $\R^N$, we set
\begin{align}
   \label{eq:65}
   \nabla^m u&=\left\{\frac{\partial ^mu}{\partial
x_{i_1}\cdots\partial
  x_{i_m}}\right\}_{i_1,\ldots,i_m\in I_N},\\
\label{eq:2}
  |\nabla^mu|&=\left\{\sum_{i_1,\ldots,i_m\in I_N}\left(\frac{\partial ^mu}{\partial x_{i_1}\cdots\partial x_{i_m}}\right)^2\right\}^{1/2}.
\end{align}
We will use the same notation for functions in $W^{m,1}(\Omega)$.
\par A simple, yet important feature of our analysis is the invariance
    of the norm in \eqref{eq:2} with respect to the orthogonal
    group $O(N)$. This is the content of the next Lemma.
    \begin{lemma}
      \label{lem:inv}
For $u\in C^m(B_R(0)\setminus\{0\})$, $m\ge1$, and $A\in O(N)$ set $u_A(x)=u(Ax)$. Then
 \begin{equation}
   \label{eq:3}
   |\nabla^m u_A(x)|=|(\nabla^mu)(Ax)|,~\forall x\in
   B_R(0)\setminus\{0\}\,.
 \end{equation}
The identity \eqref{eq:3} holds a.e.~on $\R^N$ for $u\in
W^{m,1}(\R^N)$. In particular, for any $u\in W^{m,1}(\R^N)$ and $A\in O(N)$, $u_A(x):=u(Ax)$ satisfies
 \begin{equation}
   \label{eq:28}
   \int_{\R^N} |\nabla^m u_A|=\int_{\R^N} |\nabla^m u|\,.
 \end{equation}
    \end{lemma}
 \begin{proof}
  Put $y=Ax$. From the basic formula
  \begin{equation*}
    \frac{\partial}{\partial x_i}\big(u(Ax)\big)=\sum_{j=1}^N a_{j,i}\frac{\partial u}{\partial y_j}(y)
  \end{equation*}
we deduce that for every $m$-tuple
$(i_1,\ldots,i_m)\in{(I_N)}^m$ we have
\begin{equation}
  \label{eq:1}
\frac{\partial ^m u_A}{\partial
x_{i_1}\cdots\partial
  x_{i_m}}(x)=\sum_{j_1,\ldots,j_m \in I_N}a_{j_1,i_1}a_{j_2,i_2}\cdots a_{j_m,i_m}\frac{\partial ^mu}{\partial
y_{j_1}\cdots\partial
  y_{j_m}}(y)\,.
\end{equation}
Taking the square of \eqref{eq:1} and summing over all $m$-tuples yields
\begin{multline}
\label{eq:4}
  |\nabla^m u_A(x)|^2=\sum_{i_1,\ldots,i_m \in I_N}\left(\frac{\partial ^m u_A}{\partial
x_{i_1}\cdots\partial
x_{i_m}}(x)\right)^2=\\\sum_{\substack{i_1,\ldots,i_m
\in I_N \\j_1,\ldots,j_m\in I_N\\
k_1,\ldots,k_m\in I_N}}
a_{j_1,i_1}a_{j_2,i_2}\cdots a_{j_m,i_m}
a_{k_1,i_1}a_{k_2,i_2}\cdots a_{k_m,i_m} \left(\frac{\partial ^mu}{\partial
y_{j_1}\cdots\partial
  y_{j_m}}(y)\right)\left(\frac{\partial ^mu}{\partial
y_{k_1}\cdots\partial y_{k_m}}(y)\right)\\
=\sum_{\substack{j_1,\ldots,j_m\in\{1,\ldots,N\}\\
k_1,\ldots,k_m\in\{1,\ldots,N\}}}\left\{\Prod_{s=1}^m \left(\sum_{i_s=1}^N
a_{j_s,i_s}a_{k_s,i_s}\right)\right\}\left(\frac{\partial ^mu}{\partial
y_{j_1}\cdots\partial
  y_{j_m}}(y)\right)\left(\frac{\partial ^mu}{\partial
y_{k_1}\cdots\partial y_{k_m}}(y)\right).
\end{multline}
Next we notice that since $A$ is orthogonal we have
\begin{equation*}
  \sum_{i_s=1}^Na_{j_s,i_s}a_{k_s,i_s}=\delta_{j_s,k_s},~s=1,\ldots,m.
\end{equation*}
Using it in \eqref{eq:4} gives
\begin{equation*}
  |\nabla^m
  u_A(x)|^2=\sum_{j_1,\ldots,j_m\in I_N}\left(\frac{\partial ^mu}{\partial
y_{j_1}\cdots\partial
  y_{j_m}}(y)\right)^2=|(\nabla^mu)(y)|^2,
\end{equation*}
and \eqref{eq:3} follows.
\par The statement about $u\in W^{m,1}(\R^N)$ and then the equality \eqref{eq:28} follow from
\eqref{eq:3} and the density of $C^\infty_c(\R^N)$ in $W^{m,1}(\R^N)$. 
 \end{proof}
Part (ii) of the next Corollary was proved in \cite{mss}; we present a
more elementary proof. It is based on part (i) which in turn follows
 immediately from \rlemma{lem:inv}. 
 \begin{corollary}
   \label{cor:rad}
$~$\\(i) If $u\in C^m(B_R(0)\setminus\{0\})$ or $u\in W^{m,1}(B_R(0))$, $m\ge1$, is radial then the
function $x\mapsto |\nabla^m u(x)|$ is also radial.\\
(ii) For every $m\geq 1$ there exists a positive constant $\ell^m_N$
such that
\begin{equation}
  \label{eq:5}
  |\nabla^m \log |x||=\frac{\sqrt{\ell^m_N}}{|x|^{m}},\quad\quad x\in\R^N\setminus\{0\}.
\end{equation}
 \end{corollary}
 \begin{proof}
   (i) This is a direct consequence of \rlemma{lem:inv} which gives in
   our case 
\begin{equation*}
|\nabla^m u(x)|=|\nabla^m u(Ax)|,\quad \forall
x\in\R^N\setminus\{0\},\,\forall A\in O(N).
\end{equation*}
(ii) By (i), $|\nabla^m \log |x||$ is a radial function. Since each
derivative $\frac{\partial ^m \log|x|}{\partial x_{i_1}\cdots\partial
  x_{i_m}}$ is clearly homogenous of order $-m$, the same is true for
$|\nabla^m \log |x||$. Thus, the function $|\nabla^m \log |x||$ is
necessarily of the form \eqref{eq:5}. 
 \end{proof}
For the case $m=N$ we use the shorthand $\ell_N=\ell^N_N$ (see
\eqref{eq:63}).
\begin{remark}
  \label{rem:explicit} The authors of \cite{mss} computed an 
  {\em explicit} expression for $\ell_N^m$:
  \begin{equation}
    \label{eq:6}
    \ell_N^m=m!\sum_{l=0}^{\left \lfloor{m/2}\right \rfloor}
    (m-2l)!l!\left(\frac{N-3}{2}+l\right)_l
\left(\sum_{n=\left \lceil{m/2}\right
    \rceil}^{m-l}2^{2n-m+l}\frac{(-1)^n}{2n}{\binom{n}{m-n}}{\binom{m-n}{l}}\right)^2\,,
  \end{equation}
with the notation
\begin{equation*}
  (\nu)_k=
  \begin{cases}
    \Prod_{j=0}^{k-l}(\nu-j)&\text{ for }\nu\in\R, k\in\N,\\
   1 &\text{ for }\nu\in\R, k=0.
  \end{cases}
\end{equation*}
 The first values of $\ell_N$ are: $\ell_1=1,\ell_2=2,\ell_3=28$,
 which by \rth{th:main} imply
 \begin{equation*}
   K_1=\frac{1}{2},~K_2=\frac{1}{2\pi\sqrt{2}},~K_3=\frac{1}{4\pi\sqrt{28}}\;.
 \end{equation*}

\end{remark}
\section{Proof of \rth{th:main}}
The proof of \rth{th:main} is divided into two parts. In the
first part we compute the value of $K_N$ and in the second we prove
that equality cannot hold in \eqref{eq:64}, unless $u\equiv 0$.
\subsection{The value of $K_N$}
\label{subsec:K_N}
The main ingredient of the proof of \rth{th:main} is the
identification of $\log|x|$ as a fundamental solution of a certain
operator of order $2N$:
\begin{proposition}
  \label{prop:prop}
For all  $N\geq 2$ we have, in the sense of distributions:
\begin{equation}
  \label{eq:8}
  (-1)^N \sum_{i_1,\ldots,i_N\in I_N} \frac{\partial ^N}{\partial x_{i_1}\cdots\partial
  x_{i_N}}\left(|x|^N \frac{\partial ^N\log|x|}{\partial x_{i_1}\cdots\partial
  x_{i_N}}\right)=-\ell_N\SNO\delta_0\,.
\end{equation}
\end{proposition}
\begin{proof}
  The function
  \begin{equation}
\label{eq:10}
    F(x):=(-1)^N \sum_{i_1,\ldots,i_N\in I_N} \frac{\partial ^N}{\partial x_{i_1}\cdots\partial
  x_{i_N}}\left(|x|^N \frac{\partial ^N\log|x|}{\partial x_{i_1}\cdots\partial
  x_{i_N}}\right)
  \end{equation}
 clearly belongs to $C^\infty(\R^N\setminus\{0\})$. 
 We claim that $F$ is a {\em radial}
 function. Indeed, A similar
  computation to the one used in the proof of \rlemma{lem:inv} shows
  that for a smooth $u$ on either $\R^N$ or $\R^N\setminus\{0\}$ and
  $A\in O(N)$, the function $u_A(x)=u(Ax)$ satisfies
  \begin{equation*}
    \sum_{i_1,\ldots,i_N\in I_N} \frac{\partial ^N}{\partial x_{i_1}\cdots\partial
  x_{i_N}}\left(|x|^N \frac{\partial ^N u_A}{\partial x_{i_1}\cdots\partial
  x_{i_N}}\right)(x)=\sum_{i_1,\ldots,i_N\in I_N} \frac{\partial ^N }{\partial y_{i_1}\cdots\partial
  y_{i_N}}\left(|y|^N \frac{\partial ^N u}{\partial y_{i_1}\cdots\partial
  y_{i_N}}\right)(y),
  \end{equation*}
 with $y=Ax$. Since $\log|x|$ is radial, this implies that
 $F(Ax)=F(x)$, whence $F$ is radial.

\par  From \eqref{eq:10} it is clear that $F$ is homogenous of
 degree $-N$. Therefore it must be of the form
 \begin{equation}
   \label{eq:11}
   F(x)=c|x|^{-N},
 \end{equation}
 for some constant $c\in\R$. We claim that $c=0$. 
\par Assume by contradiction that $c\neq 0$. Fix  a
function $\varphi\in C^\infty[0,\infty)$ taking values in $[0,1]$ and satisfying
\begin{equation}
  \label{eq:12}
  \varphi(t)=
  \begin{cases}
    0& \text{ for }t\in[0,1/2],\\
     1& \text{ for }t\geq 1,
  \end{cases}
\end{equation}
 and then, for any $\delta\in(0,1)$ let
 $\varphi_\delta(t)=\varphi(t/\delta)$. Fix also a function $\zeta\in
 C^\infty_c(\R^N)$, taking values in $[0,1]$ and satisfying $\zeta=1$ on $B_1(0)$ and $\zeta=0$ on
 $\R^N\setminus B_2(0)$. Finally, define
 $v_\delta(x)=\zeta(x)\varphi_\delta(|x|)$ which belongs to
 $C^\infty_c(\R^N)$. Since $v_\delta\equiv 1$ on $\{\delta<|x|<1\}$ and $|\nabla^N \varphi_\delta(|x|)|\leq
 C/\delta^N$ for $|x|\leq\delta$ we have
 \begin{equation}
   \label{eq:13}
   \int_{\R^N}|\nabla^Nv_\delta|\leq C,~\text{ uniformly in }\delta.
 \end{equation}
On the other hand, by \eqref{eq:5},
\begin{equation}
\label{eq:26}
|x|^N  \big|\frac{1}{\sqrt{\ell_N}}\nabla^N\log|x|\big|=1,
\end{equation}
whence
\begin{equation}
  \label{eq:9}
  \int_{\R^N}|\nabla^Nv_\delta|\geq
  \int_{\{\delta/2<|x|<2\}}|\nabla^Nv_\delta|\geq 
\left|\frac{1}{\sqrt{\ell_N}}\int_{\{\delta/2<|x|<2\}}\left(\nabla^Nv_\delta\right)\cdot\left(|x|^N\nabla^N\log|x|\right)\right|\,.
\end{equation}
Applying integration by parts 
to the integral on  the R.H.S.~of \eqref{eq:9} and using
\eqref{eq:11} gives
\begin{multline}
  \label{eq:14}
\frac{1}{\sqrt{\ell_N}}\int_{\{\delta/2<|x|<2\}}\left(\nabla^Nv_\delta\right)\cdot\left(|x|^N\nabla^N\log|x|\right)=\frac{1}{\sqrt{\ell_N}}\int_{\{\delta/2<|x|<2\}}Fv_\delta\\=
\frac{c}{\sqrt{\ell_N}}\int_{\{\delta/2<|x|<2\}}\frac{v_\delta}{|x|^N}=\frac{c}{\sqrt{\ell_N}}\int_{\{\delta<|x|<1\}}\frac{1}{|x|^N}+O(1)=
\frac{c}{\sqrt{\ell_N}}\,\SNO\log(1/\delta)+O(1)\,,
\end{multline}
where $O(1)$ denotes a bounded quantity,
 uniformly in $\delta$.
Combining \eqref{eq:9}--\eqref{eq:14} with \eqref{eq:13} leads to a
contradiction for $\delta$ small enough, whence $c=0$ as claimed.
\par From the above we deduce that the {\em distribution} 
\begin{equation}
\label{eq:15}
  \mathcal{F}:=(-1)^N \sum_{i_1,\ldots,i_N\in I_N} \frac{\partial ^N}{\partial x_{i_1}\cdots\partial
  x_{i_N}}\left(|x|^N \frac{\partial ^N\log|x|}{\partial x_{i_1}\cdots\partial
  x_{i_N}}\right)\in \mathcal{D}'(\R^N)
\end{equation}
satisfies $\supp(\mathcal{F})\subset\{0\}$. By a celebrated theorem of
L.~Schwartz~\cite{sch} it follows that
\begin{equation}
  \label{eq:16}
  \mathcal{F}=\sum_{j=1}^Lc_jD^{\alpha_j} \delta_0\,,
\end{equation}
for some multi-indices  $\alpha_1,\ldots,\alpha_L$.
But by \eqref{eq:5} each term on the R.H.S.~of \eqref{eq:15} can be
written as
\begin{equation*}
 \frac{\partial ^N}{\partial x_{i_1}\cdots\partial
  x_{i_N}}\left(|x|^N \frac{\partial ^N\log|x|}{\partial x_{i_1}\cdots\partial
  x_{i_N}}\right)=\frac{\partial}{\partial x_{i_1}}\left(\frac{\partial ^{N-1}}{\partial x_{i_2}\cdots\partial
  x_{i_N}}\left(|x|^N \frac{\partial ^N\log|x|}{\partial x_{i_1}\cdots\partial
  x_{i_N}}\right)\right):=\frac{\partial g}{\partial x_{i_1}},
\end{equation*}
with $g$ satisfying $|g(x)|\leq C/|x|^{N-1}$, so that $g\in
L^1_{\text{loc}}(\R^N)$. It follows that $\mathcal{F}$ in
\eqref{eq:15} is a sum of first derivatives of functions in
$L^1_{\text{loc}}$, whence for some $\mu\in\R$,
\begin{equation}
  \label{eq:41}
  \mathcal{F}=\mu\delta_0\,.
\end{equation}
\par It remains to determine the value of $\mu$ in \eqref{eq:41}. For
that matter we will use a family of test functions
$\{u_\varepsilon\}$, 
for small $\varepsilon>0$. 
Let $\varphi_\varepsilon(t)=\varphi(t/\varepsilon)$ with $\varphi$
given by \eqref{eq:12} and define  on $[0,\infty)$,
$f_\varepsilon(t)=(-\log\varepsilon)-\int_{\varepsilon}^t
\frac{\varphi_\varepsilon (s)}{s}\,ds$. 
Finally, let
 $u_\varepsilon(x)=\zeta(x)f_\varepsilon(|x|)$ on $\R^N$, where
 $\zeta$ is the same function used in the proof of \rprop{prop:prop}
 (recall $\zeta=1$ on $B_1(0)$ while $\zeta=0$ outside $B_2(0)$). It
 is easy to verify  that $u_\varepsilon\in C^\infty_c(\R^N)$
 (in fact, $\text{supp}(u_\varepsilon)\subset B_2(0)$) and it satisfies:
 \begin{align}
\|u\|_{L^\infty(\R^N)}=u_\varepsilon(0)&=\log(1/\varepsilon)+O(1),\label{eq:fl1}\\
   u_\varepsilon(x)&=\log(1/|x|) \text{ on }B_1\setminus B_\varepsilon,\label{eq:fl2}\\
\big\||\nabla^k u_\varepsilon|\big\|_{L^\infty(B_\varepsilon)}&=
  O(1)\cdot \varepsilon^{-k},\;1\leq k\leq N,\label{eq:fl3}\\
\big\||\nabla^k u_\varepsilon|\big\|_{L^\infty(\R^N\setminus B_1)}&=O(1), \;1\leq k\leq N,\label{eq:fl4}
 \end{align}
 where $O(1)$ stands for a quantity which is bounded uniformly in $\varepsilon$.


Since $u_\varepsilon\in
 C^\infty_c(\R^N)$, we get by the definition of $\mathcal{F}$ (see
 \eqref{eq:15}) and \eqref{eq:41} that
 \begin{equation}
   \label{eq:21}
   \mu u_\varepsilon(0)=\int_{\R^N}|x|^N\left(\nabla^N
     u_\varepsilon\right)\cdot\left(\nabla^N\log |x|\right)\,.
 \end{equation}
By \eqref{eq:fl2}--\eqref{eq:fl4} 
 we get for the R.H.S.~of \eqref{eq:21},
\begin{multline}
  \label{eq:23}
  \int_{\R^N}|x|^N\left(\nabla^N
     u_\varepsilon\right)\cdot\left(\nabla^N\log |x|\right)=-\int_{\{\varepsilon<|x|<1\}}|x|^N\left(\nabla^N
     \log |x|\right)\cdot\left(\nabla^N\log |x|\right)+O(1)\\
=-\ell_N \int_{\{\varepsilon<|x|<1\}}\frac{dx}{|x|^N}+O(1)=-\ell_N\SNO(-\log \varepsilon)+O(1)\,,
\end{multline}
 where we also used \eqref{eq:26}. Using \eqref{eq:21}--\eqref{eq:23} in
 conjunction with \eqref{eq:fl1}  yields
$\mu=-\ell_N\SNO$, as claimed.
\end{proof}
\begin{proof}[Proof of part (i) of \rth{th:main}]
  Clearly it is enough to consider $u\in C^\infty_c(\R^N)$ and without
  loss of generality we may assume $u(0)=\|u\|_\infty$. By
  \eqref{eq:8},
  \begin{equation}
    \label{eq:25}
    \ell_N\SNO u(0)=-\int_{\R^N}|x|^N\left(\nabla^N
     u\right)\cdot\left(\nabla^N\log |x|\right)\,.
  \end{equation}
 Using \eqref{eq:26} to bound the R.H.S.~of \eqref{eq:25} from above
 (as in \eqref{eq:9}) yields
 \begin{equation*}
    \ell_N\SNO u(0)\leq \sqrt{l_N} \int_{\R^N}|\nabla^Nu|,
 \end{equation*}
whence
\begin{equation}
  \label{eq:27}
  K_N\leq \left(\sqrt{\ell_N}\,\SNO\right)^{-1}.
\end{equation}
\par To prove that equality holds in \eqref{eq:27} it suffices to
consider $u_\varepsilon$ constructed in the course of the proof of
\rprop{prop:prop}.
Indeed, the arguments used there yield
\begin{equation}
  \label{eq:29}
  \int_{\R^N}|\nabla^Nu_\varepsilon|=\int_{\{\varepsilon<|x|<1\}}|\nabla^Nu_\varepsilon|+O(1)=
\int_{\{\varepsilon<|x|<1\}}|\nabla^N\log|x||+O(1)=\sqrt{\ell_N}\,\SNO(-\log\varepsilon)+O(1),
\end{equation}
 which in conjunction with \eqref{eq:fl1} gives 
 \begin{equation*}
   \lim_{\varepsilon\to 0^{+}}\frac{\int_{\R^N}|\nabla^Nu_\varepsilon|}{u_\varepsilon(0)}=\sqrt{\ell_N}\,\SNO.
 \end{equation*}
This clearly implies equality in \eqref{eq:27}.
\end{proof}
\subsection{Nonexistence of an optimizer in \eqref{eq:64}}
\label{subsec:no-opt}
\begin{proof}[Proof of \rth{th:main}(ii)]
  Looking for contradiction, assume  that for some $N\geq2$ there exists $u\in
 W^{N,1}(\R^N)$, $u\neq 0$, for which equality holds in
 \eqref{eq:64}. We may assume without loss of generality that
 $u(0)=\|u\|_\infty$. 
\par We first show that such $u$ can be assumed {\em
   radial}. Indeed, notice that for every $A_1,A_2\in O(N)$, the
 function $v(x):=(u(A_1x)+u(A_2x))/2$ satisfies
 $v(0)=\|v\|_\infty=\|u\|_\infty$ and 
 \begin{equation*}
   \int_{\R^N}|\nabla^N v|\leq \frac{1}{2}\left(\int_{\R^N}
     |\nabla^N u(A_1x)|+\int_{\R^N} |\nabla^N
     u(A_2x)|\right)=\int_{\R^N}|\nabla^N u|,
 \end{equation*}
 where in the last equality we used \eqref{eq:28}. It follows that $v$
 too realizes equality in \eqref{eq:64}. We can apply the same
 principle also for {\em continuous averaging}. Indeed, the function
 \begin{equation}
   \label{eq:30}
   \widetilde u(x):=\int_{SO(N)} u(Ax)\,dA,
 \end{equation}
where the integration is with respect to the (normalized) Haar measure on $SO(N)$
(see \cite{hr}), belongs to
$W^{N,1}(\R^N)$ and satisfies $\widetilde u(0)=\|\widetilde
u\|_\infty=\|u\|_\infty$ and  
\begin{equation*}
   \int_{\R^N}|\nabla^N \widetilde u|\leq \int_{\R^N}|\nabla^N u|.
 \end{equation*}
  Hence, equality must hold in the last inequality and $\widetilde u$
  is a (nontrivial) radial function for which equality holds in
  \eqref{eq:64}.
\par Let $\{u_\varepsilon\}_{\varepsilon>0}\subset C^\infty_c(\R^N)$ satisfy $u_\varepsilon\to
\tu$ in $W^{N,1}(\R^N)$, whence also in the uniform norm on
$\R^N$. Since \eqref{eq:25} holds for $u=u_\varepsilon$, passing to
the limit yields
\begin{equation}
    \label{eq:31}
    \ell_N\SNO\tu(0)=-\int_{\R^N}|x|^N\left(\nabla^N
     \tu\right)\cdot\left(\nabla^N\log |x|\right)\,.
  \end{equation}
Applying the Cauchy-Schwarz inequality to the integrand on the
R.H.S.~of \eqref{eq:31} in conjunction with \eqref{eq:63} and our
assumption that equality holds in \eqref{eq:64} for $\tu$  yields
\begin{equation}
    \label{eq:32}
    \ell_N\SNO\tu(0)=-\int_{\R^N}|x|^N\left(\nabla^N
     \tu\right)\cdot\left(\nabla^N\log |x|\right)\leq
   \sqrt{\ell_N}\int_{\R^N}|\nabla^N\tu|=\ell_N \SNO\tu(0).
  \end{equation} 
It follows from \eqref{eq:32} that $\nabla^N
\tu(x)\parallel\nabla^N\log|x|$, a.e.~in $\R^N$. Since both
$|\nabla^N\log|x||$ and  $|\nabla^N\tu(x)|$ are radial (see
\rcor{cor:rad}), it follows that there exists a function $a(t)$
such that
\begin{equation}
  \label{eq:33}
  \nabla^N \tu(x)=a(|x|) \nabla^N \log |x|,\;\text{ a.e.~on }\R^N.
\end{equation}

\par For a smooth {\em radial } function $u$ on $\R^N$ we compute, introducing the variable $s=r^2/2=|x|^2/2$,
\begin{equation*}
  u_{x_i}=\left(\frac{du}{ds}\right)\,x_i,~u_{x_ix_i}=\left(\frac{d^2u}{ds^2}\right) x_i^2+\left(\frac{du}{ds}\right) ~\text{ and
  }~u_{x_ix_j}=\left(\frac{d^2u}{ds^2}\right)x_ix_j~(i\neq j).
\end{equation*}
Here and in the sequel, with a slight abuse of notation, we will
consider a radial function $u(x)$ also as a function of $s$.
 A simple induction  shows that for any multi-index
 $\alpha=(\alpha_1,\ldots,\alpha_N)$ with
 $|\alpha|=\sum_{j=1}^N\alpha_j=m$ we have
 \begin{equation}
   \label{eq:34}
   \frac{\partial^m u}{\partial x_1^{\alpha_1}\cdots \partial
     x_N^{\alpha_N}}(x)= 
\sum_{i=0}^{\left\lfloor{m/2}\right\rfloor} \left(\frac{d^{m-i}u}{d
  s^{m-i}}\right) P^{(\alpha)}_{m-2i}(x_1,\ldots,x_N)\,,
 \end{equation}
 where each $P^{(\alpha)}_{m-2i}$ is either an homogenous polynomial of
 degree $m-2i$ with positive integer coefficients, or the zero polynomial. It follows from \eqref{eq:34}
that the tensor $\nabla^mu$ can be written as
\begin{equation}
  \label{eq:35}
  \nabla^m u(x)=\sum_{i=0}^{\left\lfloor{m/2}\right\rfloor} \left(\frac{d^{m-i}u}{d
  s^{m-i}}\right) {\mathcal P}_{m-2i}(x_1,\ldots,x_N),
\end{equation}
 where each ${\mathcal P}_{m-2i}$ is a tensor whose nonzero elements
 are taken from the set 
\begin{equation*}
\{P^{(\alpha)}_{m-2i}\,:\,|\alpha|=m,0\leq
 i\leq \left\lfloor{m/2}\right\rfloor\}\,.
\end{equation*}
 We claim that none of the tensors  $\{{\mathcal
   P}_{m-2i}\}_{i=0}^{\left\lfloor{m/2}\right\rfloor}$ is the zero
 tensor. Indeed, this follows from the simple observation that for each fixed $j\in
 I_N$ we have
 \begin{equation*}
   \frac{\partial^m u}{\partial x_j^m}(x)=\sum_{i=0}^{\left\lfloor{m/2}\right\rfloor} b_{i}\left(\frac{d^{m-i}u}{d
  s^{m-i}}\right) x_j^{m-2i},
 \end{equation*}
 with positive integer coefficients $\{b_{i}\}$.
 Now we can rewrite \eqref{eq:35} as 
\begin{equation}
  \label{eq:36}
  \nabla^m u(x)=\sum_{i=0}^{\left\lfloor{m/2}\right\rfloor} \left(\frac{d^{m-i}u}{d
  s^{m-i}}\right) (2s)^{m/2-i}{\mathcal P}_{m-2i}(\tilde
x_1,\ldots,\tilde x_N),
\end{equation}
 with $(\tilde x_1,\ldots,\tilde x_N)=(1/|x|) (x_1,\ldots,x_N)\in{\mathbb S}^{N-1}$.
 Of course, the above formulas continue to hold when we replace the smooth $u$ by
 a function belonging to $W^{m,1}$. 
\par Going back to \eqref{eq:33}, using \eqref{eq:36} for $u=\tu$ and
$u=\log|x|$, we conclude that
\begin{equation}
\label{eq:37}
  0=\sum_{i=0}^{\left\lfloor{N/2}\right\rfloor} \left(\frac{d^{N-i}\tu(x)}{d
  s^{N-i}}-a(|x|) \frac{d^{N-i}\log|x|}{d
  s^{N-i}}\right) (2s)^{N/2-i}{\mathcal P}_{N-2i}(x/|x|),\quad\text{ a.e.~on }\R^N.
\end{equation}
Since for $i_1\neq i_2$ the monomials in the components of ${\mathcal P}_{N-2i_1}$ and
${\mathcal P}_{N-2i_2}$ have different degrees, it follows from
\eqref{eq:37} that
\begin{equation}
  \label{eq:38}
  \frac{d^{N-i}\tu(x)}{ds^{N-i}}=a(|x|) \frac{d^{N-i}\log|x|}{ds^{N-i}}\,,\quad\text{a.e.~on }\R^N,~i=0,\ldots, \left\lfloor{N/2}\right\rfloor.
\end{equation}
Using \eqref{eq:38} for $i=i_0:=\left\lfloor{N/2}\right\rfloor$ and
$i=i_0-1$ yields that $\tv:=\frac{d^{N-i_0}\tu}{ds^{N-i_0}}$ satisfies
\begin{equation}
  \label{eq:39}
  \frac{\frac{d\tv}{ds}}{\tv}=\frac{d^{N-i_0+1}\log|x|/ds^{N-i_0+1}}{d^{N-i_0}\log|x|/ds^{N-i_0}}=-\frac{N-i_0}{s},\quad\text{a.e.~for } s\in(0,\infty).
\end{equation}
 Integrating \eqref{eq:39} gives $\tv=\frac{c}{s^{N-i_0}}$ for some
   constant $c$, whence
   \begin{equation}
     \label{eq:40}
     \tu={\tilde c}\log s+Q_{N-i_0-1}(s),
   \end{equation}
 where $Q_{N-i_0-1}$ is a polynomial of degree less or equal to  $N-i_0-1$ and $\tilde c$ is another constant. For $\tu$ as
 in \eqref{eq:40}, the 
 requirements $\tu,|\nabla^N\tu|\in L^1(\R^N)$ clearly impose
 $\tu=0$. Contradiction.
\end{proof}
\begin{remark}
  It was shown in \cite{hn} that for $N=1$ the function
  $u(x)=e^{-|x|}$ satisfies $u(0)=\|u\|_\infty=(1/2)\int_\R|u'|$, that is, equality
  holds in \eqref{eq:64}. In fact, this is
  true for any $u\in W^{1,1}(\R)$ satisfying $\sgn u'(x)=-\sgn x$,
  a.e.~on $\R$.
\end{remark}

\end{document}